\documentclass[a4paper, 12pt, twoside,openright]{article}

\usepackage{enumerate}
\usepackage{enumitem}
\usepackage[section]{placeins}	
\usepackage{epigraph}
\usepackage[font={small}]{caption}
\usepackage{appendix}

\usepackage{amsmath,amssymb,mathrsfs}
\usepackage{amsfonts}			
\usepackage{nicefrac}			

\usepackage{bbm} 

\usepackage{stmaryrd}
\usepackage{calc,ifthen,xspace}

\usepackage[all,cmtip]{xy}
\usepackage{array}
\usepackage{multirow}
\usepackage{booktabs}
\usepackage{colortbl}
\usepackage{tabularx}
\usepackage{multirow}
\usepackage{threeparttable}

\usepackage{graphicx}			
\usepackage{subcaption}
\usepackage{pdfpages}
\usepackage{rotating}
	
\usepackage{tikz}

\usepackage{tikz-cd}

\usepackage{xspace}
\usepackage{textcomp}
\usepackage{array}
\usepackage{hyphenat}

\usepackage[pdftex,pdfborder={0 0 0},
			colorlinks=true,
			linkcolor=blue,
			citecolor=red,
			pagebackref=true,
			]{hyperref}	

\usepackage[top=2.5cm, bottom=2cm, left=3cm, right=2.5cm,
			headheight=15pt]{geometry}

\AtBeginDocument{}

\renewcommand{\leq}{\ensuremath{\leqslant}} 
\renewcommand{\geq}{\ensuremath{\geqslant}}

\DeclareMathOperator{\Hom}{Hom}
\DeclareMathOperator{\Aut}{Aut}

\DeclareMathOperator{\Fct}{Fct} 

\DeclareMathOperator{\Ran}{Ran}

\usepackage{amsthm}

\makeatletter
\newtheorem*{rep@theorem}{\rep@title}
\newcommand{\newreptheorem}[2]{%
\newenvironment{rep#1}[1]{%
 \def\rep@title{#2 \ref{##1}}%
 \begin{rep@theorem}}%
 {\end{rep@theorem}}}
\makeatother

\theoremstyle{plain}
\newtheorem{theo}{Theorem}[section]
\newreptheorem{theo}{Theorem}
\newtheorem{fact}[theo]{Fact} 
\newtheorem{lem}[theo]{Lemma}
\newtheorem{prop}[theo]{Proposition}
\newreptheorem{prop}{Proposition}

\newtheorem{cor}[theo]{Corollary}
\newreptheorem{cor}{Corollary}

\newreptheorem{pb}{Problem}

\numberwithin{equation}{theo}

\theoremstyle{definition}
\newtheorem{defi}[theo]{Definition}

\newcommand{\Enstq}[2]{\left\{\ #1\ \middle|\ #2\ \right\}}    

\let\oldpagenumbering\pagenumbering
\renewcommand{\pagenumbering}[1]{%
	\cleardoublepage
	\oldpagenumbering{#1}
}

\author{Jacques \scshape{Darn\'e}}
\title{Co-induced actions for topological and filtered groups}
\date{December 1, 2018}

\hypersetup{
    pdfauthor={Jacques Darn\'e},
    pdfsubject={SCF is LACC},
    pdftitle={Co-induced actions for topological and filtered groups},
    pdfkeywords={category theory, group theory, filtered groups, Lie algebras}
}

\begin{document}

\maketitle

\begin{abstract}
In this note, we show that the category $\mathcal{SCF}$ introduced in \cite{Darné} admits co-induced actions, which means that it is Locally Algebraically Cartesian Closed \cite{Gray, Bourn-G}. We also show that some co-induction functors exist in the category of topological groups, and that a \emph{convenient} category of topological groups is LACC.
\end{abstract}

\section*{Introduction}

The present paper deals with the study of actions in some categories of topological and filtered groups. The author's main motivation is the study of strongly central filtrations on groups (also called $N$-series). These occur in several contexts. For instance:
\begin{itemize}
\item On the Torelli subgroup of the automorphisms of free groups, there are two such filtrations defined in a canonical way ; the \emph{Andreadakis problem}, still very much opened, asks the question of the difference between them.
\item On the Torelli subgroup of the mapping class groups of a surface, the \emph{Johnson filtration} is a $N$-series. The difference between this filtration and the lower central series is linked to invariants of $3$-manifolds, such as the \emph{Casson invariant}.
\item On the pure braid groups or the pure welded braid groups, such filtrations appear in the study of \emph{Milnor invariants} and \emph{Vassiliev invariants}.
\end{itemize}
Considering strongly central filtrations as a category has led to a better understanding of phenomena appearing in these various contexts, in particular the role of Johnson morphisms, or semi-direct product decompositions of certain associated Lie algebras. At the heart of this work lies the study of \emph{actions} in this category. 
Our main result here is a further step in that direction:

\begin{reptheo}{SCF_is_LACC}
The category of strongly central filtrations admits co-induction functors along any morphism $\alpha: E_* \rightarrow B_*$, that is, a right adjoint to the restriction of $B_*$-actions along $\alpha$, for all $\alpha$.
\end{reptheo}

Moreover, our methods can be adapted to the case of topological groups, to show:

\begin{reptheo}{coinduction_between_lc_groups}
In the category of topological groups, there are co-induction functors along any morphism between locally compact groups.
\end{reptheo}

If we restrict our attention to a \emph{convenient} category of topological spaces, then more is true:

\begin{reptheo}{top_groups_is_LACC}
A \emph{convenient} category of topological groups admits co-induction functors along all morphisms.
\end{reptheo}

This last result can also be seen as an enriched version of the existence of co-induction in the category of groups, obtained from Kan extensions.

\section*{Strongly central filtrations and actions}

\subsection*{Strongly central filtrations}

We recall the definition of the category $\mathcal{SCF}$ introduced in \cite{Darné}.
\begin{defi}
A strongly central filtration $G_*$ is a nested sequence of groups  $G_1 \supseteq G_2 \supseteq G_3 \cdots$ such that $[G_i, G_j] \subseteq G_{i+j}$ for all $i, j \geq 1$. These filtrations are the objects of a category $\mathcal{SCF}$, where morphisms from $G_*$ to $H_*$ are those group morphisms from $G_1$ to $H_1$ sending each $G_i$ into $H_i$.
\end{defi}
Recall from \cite{Darné} that this category is complete, cocomplete, and homological, but not semi-abelian. It is also action-representative.

\subsection*{Actions}

In a protomodular category $\mathcal{C}$, an \emph{action} of an object $B$ on an object $X$ is a given isomorphism between $X$ and the kernel of a split epimorphism with a given section $\begin{tikzcd}[column sep=small] Y \ar[r, two heads] & B \ar[l, bend right] \end{tikzcd}$. Objects of $\mathcal C$ endowed with a $B$-action form a category, where morphisms are the obvious ones. This category of $B$-objects is the same as the category of split epimorphisms onto $B$ with given sections, also called \emph{points} over $B$, and is denoted by $Pt_B(\mathcal C)$.

For example, an action in the category of groups is a group action of a group on another, by automorphism (see for instance \cite{BJK}). The equivalence between $B$-groups (in the last sense) and $B$-points is given by: 
\[B \circlearrowright G \longmapsto \left(\begin{tikzcd} B \rtimes G \ar[r, two heads] & B \ar[l, bend right] \end{tikzcd}\right).\] 
The same construction allows us to identify the category of $B$-points in topological groups with the category of topological groups $G$ endowed with a topological $B$-action, \emph{i.e.\!} a group action such that $B \times G \rightarrow G$ is continuous.

Recall from \cite[Prop. 1.20]{Darné} that in $\mathcal{SCF}$, an action of $B_*$ on $G_*$ is the data of a group action of $B_1$ on $G_1$ such that $[B_i, G_j] \subseteq G_{i+j}$ for all $i, j \geq 1$, where the commutator is taken in the semi-direct product $B_1 \ltimes G_1$, that is: $[b, g] = (b \cdot g)g^{-1}$.

\subsection*{Restriction and induction functors}

Suppose that our protomodular category $\mathcal C$ admits finite limits and colimits. Let $\alpha: E \rightarrow B$ be a morphism in $\mathcal C$. We can restrict a $B$-action along $\alpha$ by pulling back (in $\mathcal C$) the corresponding epimorphism. This defines a \emph{restriction functor} (also called \emph{base-change functor}):
\[\alpha^*: Pt_B \longrightarrow Pt_E.\]
If we are given a $E$-action, we can define the induced $B$-action by pushing out (in $\mathcal C$) the corresponding section. This defines an \emph{induction functor} $\alpha_*$, which is left adjoint to $\alpha^*$. 

A question then arises naturally: is there a co-induction functor ? That is, when does the restriction functor $\alpha^*$ also have a right adjoint ? When all $\alpha$ have this property, the category is called \emph{Locally Algebraically Cartesian Closed (LACC)}. This is a rather strong condition, implying for instance algebraic coherence \cite[Th. 4.5]{VdL}. This condition has been studied for example in \cite{Gray, Bourn-G}.

\bigskip

\noindent
\textbf{Acknowledgements}: The author thanks Tim Van der Linden for asking the question which became the starting point of this work, and for inviting him in Louvain-la-Neuve, where we had the most interesting and helpful discussions. He also thanks Alan Cigoli for taking part in these discussions and making some very interesting remarks.

\tableofcontents

\section{Co-induction in the category of groups} \label{section_groups}

Our aim is to show that the category $\mathcal{SCF}$ is LACC, that is, that it admits co-induced actions. We first review the case of groups, that will be the starting point of our construction.

The construction of co-induced group actions follows easily from the following fact: if $B$ is a group, the category $Pt_B(\mathcal{G}rps)$ of $B$-groups identifies with the functor category $\Fct(B, \mathcal Grps)$, where $B$ is considered as a category with one object. Then the restriction functor $\alpha^*$ along a group morphism $\alpha: E \rightarrow B$ is given by precomposition by $\alpha$ (interpreted as a functor between the corresponding small categories). Since $\mathcal Grps$ is complete and co-complete, this functor has both a left and a right adjoint, given by left and right Kan extensions along $\alpha$.

Let us describe the right adjoint of $\alpha^*$ in this context. If $\varphi: \mathcal C \rightarrow \mathcal D$ is a functor between small categories, recall that the right Kan extension of $F: \mathcal C \rightarrow \mathcal T$ along $\varphi$ is given by the end formula:
\[\Ran_\varphi(F) = \int_{c \in \mathcal C} \hom(\mathcal D(-,\varphi(c)), F(c)),\]
where $\hom$ denote the co-tensor over $\mathcal Sets$: if $T$ is a set, and $T \in \mathcal T$, then $\hom(X,T) = T^X$ is the product of $X$ copies of $T$.

In our situation ($\varphi = \alpha$, and $F = Y$ is a $B$-group), the coend is taken over the one object $*$ of $E$. Moreover, $\hom(B(*,\alpha(*)),Y(*))$ is the group of applications from $B$ to $Y$ (whose product is defined pointwise), and the coend is the subgroup of applications $u$ satisfying $u(\alpha(e)b)=e \cdot u(b)$ for all $e \in E$ and $b \in B$, that is, the subgroup of $E$-equivariant applications $\hom_E(B,Y)$. The action of $B$ on $\hom_E(B,Y)$ is given by $(b \cdot u)(-) = u( - \cdot b)$. Thus, we recover:

\begin{prop} \textup{\cite[Th.\ 6.11]{Gray}}.
In the category of groups, co-induction along a morphism $\alpha: E \rightarrow B$ is given by
$Y \longmapsto \hom_E(B,Y),$
where $\hom_E(B,Y)$ is the group of $E$-equivariant applications from $B$ to $Y$ (with multiplication defined pointwise), on which $B$ acts by $(b \cdot u)(-) = u( - \cdot b)$.
\end{prop}

\section{Co-induction in the category \texorpdfstring{$\mathcal{SCF}$}{SCF}}

The construction described above in the category of groups seems to be very specific, as it relies on an identification between the categories of points and functor categories. Such an identification does not hold in the case of strongly central filtrations. However, we will be able to compare the categories of points with functor categories. This will allow us to use Kan extensions for constructing co-induction.

\subsection{Categories of points and functor categories}

When dealing with a strongly central filtration $G_*$, we will often omit the subscript $1$, for short, denoting the underlying group $G_1$ by $G$.

Let $B_*$ be a strongly central filtration. There is an obvious forgetful functor, recalling only that $B$ acts by automorphisms preserving the filtration:
\[\omega: Pt_{B_*}(\mathcal{SCF}) \longrightarrow \Fct(B,\mathcal{SCF}).\]
Since the compatibility conditions it forgets are only conditions on objects, not on morphisms, this functor is fully faithful.
Let $\alpha: E_* \rightarrow B_*$ be a morphism in $\mathcal{SCF}$. The restriction functor $\alpha_1^*$ between the corresponding functor categories fits into a commutative diagram:
\[\begin{tikzcd}
Pt_{B_*}(\mathcal{SCF}) \ar[r, "\omega"] \ar[d, "\alpha^*"] &\Fct(B,\mathcal{SCF}) \ar[d, "\alpha_1^*"] \\
Pt_{E_*}(\mathcal{SCF}) \ar[r, "\omega"]                    &\Fct(E,\mathcal{SCF}).
\end{tikzcd}\]
Moreover, since $\mathcal{SCF}$ is complete, $\alpha_1^*$ has a right adjoint, given by the right Kan extension $(\alpha_1)_!$. It follows from the description of limits in $\mathcal{SCF}$ \cite[Prop. 1.10]{Darné} and from the construction of Section \ref{section_groups} that $(\alpha_1)_!(Y_*)$ is $\hom_E(B,Y_*)$, that is, the filtration defined pointwise on $\hom_E(B,Y)$ (which is obviouly strongly central).

If we construct a right adjoint $t$ to the forgetful functor $\omega$ (for any $B_*$), then by composing adjunctions, $t \circ (\alpha_1)_!$ will be right adjoint to $\alpha_1^* \circ \omega$, and $t \circ (\alpha_1)_! \circ \omega$ will be right adjoint to $\alpha^*$, because $\omega$ is fully faithful:
\begin{align*}
\Hom\left(\alpha^*(X_*), Y_*\right) &= \Hom\left(\omega \circ \alpha^*(X_*), \omega (Y_*)\right) \\
                         &= \Hom\left(\alpha_1^* \circ \omega(X_*), \omega (Y_*)\right) \\
                         &= \Hom\left(X_*, t \circ (\alpha_1)_! \circ \omega (Y_*)\right).
\end{align*}

We will construct this functor $t = t^\infty (B_*, -)$ in the next section. This will finish the proof of our first main theorem:

\begin{theo}\label{SCF_is_LACC}
There are co-induction functors in $\mathcal{SCF}$. Explicitly, co-induction along $\alpha: E_* \rightarrow B_*$ is given by:
\[Y_* \longmapsto \alpha_!(Y_*) = t^\infty (B_*, \hom_E(B,Y_*)),\]
which is the largest strongly central filtration smaller than $\hom_E(B,Y_*)$ on which $B_*$ acts.
\end{theo}

\subsection{Maximal \texorpdfstring{$B_*$}{B*}-filtration}

Since the forgetful functor $\omega: Pt_{B_*}(\mathcal{SCF}) \rightarrow \Fct(B,\mathcal{SCF})$ is fully faithful, if it admits a right adjoint $t$, then the counit $\omega t(G_*) \rightarrow G_*$ has to be a monomorphism. Thus, if $G_*$ is a strongly central filtration on which $B$ acts by automorphisms preserving the filtration, we need to construct the maximum sub-object of $G_*$ such that the restricted action of $B$ satisfies the conditions defining a $B_*$-action. We will do so through a limit process, restricting to smaller and smaller subgroups endowed with smaller and smaller filtrations.

\begin{prop}\label{maximum_action}
Let $B_*$ and $G_*$ be strongly central filtrations, and let $B = B_1$ act on $G = G_1$ by group automorphisms preserving the filtration $G_*$. Then there exists a greatest one among those strongly central filtrations $H_* \subseteq G_*$ such that the action of $B$ on $G$ induces and action of $B_*$ on $H_*$. 
\end{prop}

\begin{cor}
The forgetful functor $\omega: Pt_{B_*}(\mathcal{SCF}) \rightarrow \Fct(B,\mathcal{SCF})$ has a left adjoint $t$, the filtration $t(G_*)$ being the greatest filtration constructed in the above proposition.
\end{cor}

\begin{proof}
Suppose that $B_*$ acts on $K_*$, that the group $B$ acts on $G_*$ (by filtration-preserving automorphisms), and that $f: K_* \rightarrow G_*$ is a $B$-equivariant morphism. Then $B_*$ acts on $f(K_*) \subseteq G_*$. As a consequence, $f(K_*) \subseteq t(G_*)$, so that $f$ is in fact a morphism from $K_*$ to $t(G_*)$.
\end{proof}

\begin{lem}\label{transportation}
Under the hypothesis of the proposition, the filtration defined by:
\[t_i(B_*, G_*):= \Enstq{g \in G_j}{\forall j,\ [B_j, g] \subseteq G_{i+j}}\]
is strongly central, and stable under the action of $B$.
\end{lem}

\begin{proof}
For short, denote $t_*(B_*, G_*)$ by $t_*$. These are subgroups of $G$ because of the normality of $G_{i+j}$ and the formula $[b, gg'] = [b,g]\cdot ({}^g\![b,g'])$. The three subgroups lemma applied in $B \rtimes G$ tells us that $[B_k, [t_i,t_j]]$ is contained in the normal closure of $[[B_k, t_i],t_j]$ and $[[B_k, t_j],t_i]$. These are both inside $G_{i+j+k}$, which is normal in $B \rtimes G$ because it is normal in $G$ and $B$-stable. Thus $[B_k, [t_i,t_j]] \subseteq G_{i+j+k}$, which means exactly that $[t_i,t_j] \subseteq t_{i+j}$. Each $t_i$ is also stable under the action of $B$, since if $b \in B$, using that the $B_j$ and the $G_k$ are $B$-stable, we have that $[B_j, t_i^b] = [{}^b\! B_j, t_i]^b  \subseteq [B_j, t_i]^b \subseteq G_{i+j}^b \subseteq G_{i+j}$,
thus $t_i^b \subseteq t_i$.
\end{proof}

\begin{proof}[Proof of Proposition \ref{maximum_action}]
Using Lemma \ref{transportation}, define a descending series of strongly central filtrations by:
\[\begin{cases} t^0_*(G_*):= G_*, \\ t^{l+1}_*(G_*):= t_*(B_*, t^l_*(G_*)). \end{cases}\]
Then define $t^\infty_* (G_*)$ to be their intersection. We claim that $t^\infty_* (G_*)$ is the requested filtration. Firstly, $B_*$ acts on it ; if $b \in B_j$, then:
\[\forall i, j, l,\ [B_j, t^\infty_i (G_*)] \subseteq [B_j, t^{l+1}_i (G_*)] \subseteq t^l_{i+j}(G_*),\]
so that by taking the intersection on $l$,
\[\forall i, j,\ [B_j, t^\infty_i (G_*)] \subseteq t^\infty_{i+j}(G_*).\]
Secondly, if $B_*$ acts on $H_* \subseteq G_*$, then by we show that $H_* \subseteq t^l_*(G_*)$ for all $l$: by definition, $H_* \subseteq G_* = t^0_*(G_*)$, and if $H_* \subseteq t^l_*(G_*)$ for some $l$, then:
\[\forall i, j,\ [B_i, H_j] \subseteq H_{i+j} \subseteq t^l_{i+j}(G_*),\]
proving that $H_* \subseteq t_*(B_*, t^l_*(G_*)) = t^{l+1}_*(G_*)$, and our claim. 
\end{proof}

\section{Co-induction for topological groups}

Our argument for strongly central filtrations can be adapted to the case of topological groups. The first part, about Kan extension, is exactly the same. The comparison between the category of points and functors categories, however, does only work with some restrictions : we need either the acting group to be locally compact, or the category of topological spaces to be a \emph{convenient} one, for example the category of compactly generated weakly Hausdorff spaces. The main idea behind the construction is the same as before: restricting to smaller and smaller subgroups, and refining more and more the topology.

\subsection{Categories of points and functor categories}

A topological group is the data of a group $G$ and a topology $\tau$ on it, with the usual compatibility requirements \cite[Chap. 3]{Bourbaki} ; it can be seen as a group object in the category of topological spaces. We will denote such an object by $(G,\tau)$, or only by $G$ or by $\tau$, whenever the rest of the data is clear from the context. Also, in any topological group $G$ and any point $g \in G$, we will denote by $\mathcal V_G(g)$ the set of neigbourhoods of $g$ in $G$. Note that the only topologies we will consider are the ones compatible with group structures ; expressions such as "the \emph{finest} topology satisfying..." have to be understood with this requirement in mind. Topological groups, together with continuous morphisms, form a category $\mathcal TopGrp$, which is complete, cocomplete, and homological.

If $B$ is a topological group, the obvious forgetful functor $\omega$ from $Pt_{B}(\mathcal TopGrp)$ to $\Fct(B,\mathcal TopGrp)$ is fully faithful. If $\alpha: E \rightarrow B$ is a continuous morphism, then restriction along $\alpha$ is defined in the usual way, between the category of points and between the categories of functors. The picture is exactly the same as in the case of strongly central filtrations, and the right Kan extension along $\alpha$ is defined in the same way: $\alpha_!(Y) =\hom_E(B,Y) \subseteq Y^B$, endowed with the topology inherited from the product topology on $Y^B$. We will show below (Proposition \ref{maximum_top_action}) that the forgetful functor admits a right adjoint for any locally compact $B$. Thus we will have proved:

\begin{theo}\label{coinduction_between_lc_groups}
There are co-induction functors along morphisms between locally compact groups in $\mathcal TopGrp$. Explicitly, co-induction along $\alpha: E \rightarrow B$ is given by:
\[Y \longmapsto \alpha_!(Y) = t^\infty (B, \hom_E(B,Y)),\]
which is the largest topological group endowed with an monomorphism into $\hom_E(B,Y)$, on which $B$ acts continuously.
\end{theo}
Note that it is not the largest topological subgroup, as \emph{largest} here also mean: endowed with the coarsest possible topology.

If we restrict our category of topological spaces (for example to compactly generated weakly Hausdorff spaces), then we will show (Proposition \ref{maximum_cgwh_action}) that our construction works for all group object $B$.

\begin{theo}\label{top_groups_is_LACC}
There are co-induction functors in the category of topological groups (that is, it is LACC) when the base category of topological spaces is the category $CGWH$ of compactly generated weakly Hausdorff spaces, or any \emph{convenient} category of topological spaces, is a sense made precise below.
\end{theo}

\subsection{Maximal \texorpdfstring{$B$}{B}-topology}

Suppose that $B$ is locally compact (note that we do not imply that it is locally Hausdorff). We now construct the right adjoint to the forgetful functor, in a series of lemmas leading to the proof of:

\begin{prop}\label{maximum_top_action}
For any locally compact topological group $B$, the forgetful functor from $Pt_{B}(\mathcal TopGrp)$ to $\Fct(B,\mathcal TopGrp)$ has a right adjoint.
\end{prop}

We will make use of the classical :
\begin{prop}\label{adjunction}
Let $B$ be a locally compact topological space. Then $B \times (-)$ has a right adjoint, given by $\mathcal C(B,-)$, where $\mathcal C(B,Y)$ is the set of continuous maps from $B$ to $Y$, endowed with the compact-open topology.
\end{prop}

\begin{proof}[Idea of proof]
It is easy to see that a map $B \times X \rightarrow Y$ is continuous iff $X \rightarrow \mathcal Ens (B,Y)$ takes values inside $\mathcal C(B,Y)$, and is continuous. The "only if" part crucially uses the fact that $B$ is locally compact.
\end{proof}

Consider $G$ an element of $\Fct(B,\mathcal TopGrp)$. This means that $G$ is a topological group endowed with an action of the discrete group $B$ \emph{via} continuous automorphisms. If $H \leq G$ is a subgroup, consider the map:
\[a_H: B \times H \rightarrow G\]
obtained by restriction of the action of $B$ on $G$. How can we endow $H$ with a new topology such that $a_H$ is continuous ? In order to do this, we first need to restrict to a subgroup of $G$. Indeed, if $a_H$ is continuous (for any topology on $H$), then for any $g \in H$, the map $a_H(-,g) = (-) \cdot g : B \rightarrow G$ has to be continuous. Thus we are led to the definition:
\[G_1 = \Enstq{g \in G}{(-) \cdot g \text{ is continuous}}.\]

\begin{lem}\label{G1_subgroup}
The subset $G_1$ is a subgroup of $G$, stable under the action of $B$.
\end{lem}

\begin{proof}
Consider $\bar a: g \mapsto (-) \cdot g$, the set map from $G$ to $\mathcal Ens (B,G)$ adjoint to the action of $B$ on $G$. Since $G$ is a topological group, the subset $\mathcal C(B,G)$ of continuous maps is a subgroup of $\mathcal Ens (B,G)$ (both are endowed with the pointwise law). Moreover, since $B$ is a topological group, it is $B$-stable under the $B$-action $b \cdot u = u(- \cdot b)$. Since $B$ acts on $G$ by group automorphism, $\bar a$ is a group morphism. Moreover, it is also $B$-equivariant, as one easily checks. Thus $G_1 = \bar a^{-1}(\mathcal C(B,G))$ is a $B$-stable subgroup of $G$.
\end{proof}

We want to define the coarsest topology on $G_1$ making $a: B \times G_1 \rightarrow G$ continuous. Such a topology exists, as it is the coarsest topology making $\bar a: G_1 \rightarrow \mathcal C(B,G)$ continuous, which is exactly the subspace topology on $G_1 \subseteq \mathcal C(B,G)$. We denote it by $\tau_1$. Remark that $\tau_1$ is finer than the subgroup topology on $G_1 \subseteq G$, because $a(1,-)$ is exactly the inclusion $G_1 \hookrightarrow G$, that has to be continuous when $G_1$ is endowed with $\tau_1$.

\begin{lem}\label{G1_action_continuous}
The group $B$ acts on $\mathcal C(B,G)$, whence also on $(G_1, \tau_1)$, by continuous automorphisms.
\end{lem}

\begin{proof}
The action of $b \in B$ on $\mathcal C(B,G)$ is given by pre-composition by $(-) \cdot b$, which is continuous. By functoriality of $\mathcal C(-,G)$, it is continuous. Moreover, since the group law is defined pointwise on the target, it acts \emph{via} an automorphism: \newline
$b \cdot(uv) = uv(- \cdot b) = u(- \cdot b)v(- \cdot b) = (b \cdot u)(b \cdot v).$
\end{proof}

\begin{lem}\label{Adjunction_step}
Let $B$ act topologically on a topological group $X$, and $f: X \rightarrow G$ be a continuous $B$-equivariant map. Then $f(X) \subseteq G_1$ and $f: X \rightarrow (G_1, \tau_1)$ is continuous.
\end{lem}

\begin{proof}
The map $(b,x) \mapsto f(b \cdot x)$ is the composite $B \times X \rightarrow X \overset{f}{\rightarrow} G$, so it is continuous. Its adjoint $x \mapsto f(- \cdot x)$ is also continuous and, since it coincides with $x \mapsto (-) \cdot f(x)$, it factorizes as a set map through $G_1$, which means that $f$ takes values in $G_1$. Moreover, $f$ is a continuous map from $X$ to the subspace $G_1$ of $\mathcal C(B,G)$.
\end{proof}

\begin{proof}[Proof of Proposition \ref{maximum_top_action}] 
Let $(G, \tau)$ be a topological group on which $B$ acts by continuous automorphisms. We iterate the construction $t: G \mapsto (G_1, \tau_1)$ described above. This is possible thanks to Lemmas \ref{G1_subgroup} and \ref{G1_action_continuous}. We denote by $(G_l, \tau_l)$ the $l$-th iterate $t^l(G)$, and we define $t^\infty(G)$ as the intersection $G_\infty$ of the $G_l$, endowed with the reunion $\tau_\infty$ of the $\tau_l|_{G_\infty}$. That is, $G_ \infty$ is the (topological) projective limit of the $G_l$. 

We first show that the action of $B$ on $G_\infty$ is topological, that is, that $B \times G_\infty \rightarrow G_\infty$ is continuous. This is equivalent to the map $B \times G_\infty \rightarrow G_l$ being continuous for every $l$. But this last map can be seen as the composite $B \times G_\infty \rightarrow B \times G_{l+1} \rightarrow G_l$, and these maps are continuous by construction, so the action of $B$ on $G_\infty$ is indeed topological.

Now suppose that $B$ acts topologically on $X$, and $f: X \rightarrow G$ is a continuous $B$-equivariant map. Then, thanks to Lemma \ref{Adjunction_step}, $f: X \rightarrow (G_1, \tau_1)$ is again a continuous $B$-equivariant map, and by iterating the construction, we see that $f: X \rightarrow t^l(G)$ is, for all $l$. Thus $f$ takes values in $G_\infty$ and is continuous with respect to $\tau_\infty$. Since any continuous $B$-equivariant map $f: X \rightarrow t^\infty(G)$ comes uniquely from such an $f$ (which is continuous because the injection of $G_\infty$ into $G$ is, since it is $a_{G_\infty}(1,-)$), we have showed that $t^\infty$ is the right adjoint we were looking for.
\end{proof}

\subsection{Restricting to a convenient category of spaces}

If $B$ is any topological group acting on another topological group $G$, there does not seem to be a coarsest topology $\tau_1$ on $G_1$ making $B \times G_1 \rightarrow G$ continuous, so our construction fails to produce an adjoint to the corresponding forgetful functor. However, we can change that by restricting to a \emph{convenient} category of topological spaces. We need this category $\mathcal T$ to satisfy the following four hypotheses:

\begin{itemize}
\item $\mathcal T$ is a full subcategory of topological spaces containing the one-point space.
\item $\mathcal T$ admits (small) limits.
\item $\mathcal T$ is cartesian closed.
\item If a subset $X$ of an object $T \in \mathcal T$ is given, there should be a topology $\tau$ on $X$ such that $(X, \tau)$ is in $\mathcal T$, the injection $(X, \tau) \hookrightarrow T$ is continuous, and every $f : Y \rightarrow T$ in $\mathcal T$ such that $f(Y) \subseteq X$ defines a continuous map $f : Y \rightarrow (X, \tau)$. Such a topology is called the \emph{$\mathcal T$-subspace topology} on $X$.
\end{itemize}

\begin{fact} \textup{\cite[Prop.\ 2.12, Lem.\ 2.28 and Prop.\ 2.30]{Strickland}}.
These hypotheses are satisfied if $\mathcal T = CGWH$ is the category of compactly generated Hausdorff spaces.
\end{fact}

The point, being final in $\mathcal T$, is the unit of the cartesian monoidal structure. Thus the forgetful functor to sets is $\mathcal T(*,-)$. In particular, if we denote by $\mathcal C(B,-)$ the right adjoint to $B \times(-)$, the underlying set of $\mathcal C(B,X)$ is the set of continuous maps from $B$ to $X$:
\[\mathcal T(*,\mathcal C(B,X)) = \mathcal T(B \times *, X) = \mathcal T(B, X).\]
Moreover, the underlying set of a limit is the limit of the underlying diagram to sets. Indeed, if $\mathcal D$ is a small category and $F : \mathcal D \rightarrow \mathcal T$ is a diagram, then :
\[\mathcal T(*,\lim F) = \lim \mathcal T(*, F).\]

The reader can check that the constructions of the previous paragraph work well under these hypotheses, replacing the category $\mathcal TopGrp$ by the category $\mathcal T Grp$ of group objects in $\mathcal T$. Precisely, the topology $\tau_1$ has to be the $\mathcal T$-subspace topology on $G_1 \subseteq \mathcal C(B, X)$, and $G_\infty$ has to be the limit of the $G_l$ in $\mathcal T$. Thus we can state:

\begin{prop}\label{maximum_cgwh_action}
For any topological group $B \in \mathcal T Grp$, the forgetful functor from $Pt_{B}(\mathcal TGrp)$ to $\Fct(B,\mathcal TGrp)$ has a right adjoint. 
\end{prop}

Remark that this proposition, together with the following fact, suggest that $\mathcal TGrp$ is a nice category to work with.

\begin{fact}
The category $\mathcal TGrp$ is also action-representative: a representant of actions on $G$ is the set of continuous automorphisms $\Aut(G) \subset \mathcal C(G,G)$, endowed with the $\mathcal T$-subspace topology.
\end{fact}

\subsubsection*{A remark on \texorpdfstring{$\mathcal T$}{T}-denriched categories}

Theorem \ref{top_groups_is_LACC} can be obtained directly, in a fashion similar to the construction of co-induction for groups. To do that, we use the language of $\mathcal T$-enriched categories \cite{Kelly}. The category $\mathcal TGrp$ is $\mathcal T$-enriched ; moreover, every $\mathcal T$-group $B$ can be considered as a $\mathcal T$-category with one object, and there is an obvious equivalence:
\[Pt_{B}(\mathcal TGrp) \simeq \Fct_{\mathcal T}(B, \mathcal TGrp).\]

Thus the same construction as in ordinary groups (Section \ref{section_groups}) works here, replacing Kan extensions by enriched Kan extensions. This uses the fact that $\mathcal TGrp$ is $\mathcal T$-complete (that is, it is complete and co-tensored over $\mathcal T$).

This gives an alternative proof of Theorem \ref{top_groups_is_LACC}. However, neither Theorem \ref{coinduction_between_lc_groups} nor Theorem \ref{SCF_is_LACC} fits in this machinery. Moreover, Proposition \ref{maximum_cgwh_action} is still meaningful in this context: it provides a right adjoint to the forgetful functor from the category $\Fct_{\mathcal T}(B, \mathcal TGrp)$ of enriched functors  to the category $\Fct(B, \mathcal TGrp)$ of non-enriched ones.
			
\bibliographystyle{alpha}
\bibliography{Ref_LACC}

\end{document}